\documentclass{amsart}
\usepackage{hyperref,amsmath,amsthm,verbatim,amssymb,amsfonts,amscd,graphicx,bm} % standard packages
\usepackage[utf8]{inputenc}
\usepackage{bookmark}
\newtheorem{theorem}{Theorem}[section]

\newtheorem{lemma}[theorem]{Lemma}
\newtheorem{corollary}[theorem]{Corollary}

\newtheorem{question}[theorem]{Question}

\theoremstyle{definition}
\newtheorem{definition}[theorem]{Definition}

\newtheorem{remark}[theorem]{Remark}

\newcommand*{\abs}[1]{\lvert#1\rvert}
\newcommand*{\concat}{\symbol{94}}

\newcommand*{\nbd}{\nobreakdash-\hspace{0pt}}
\DeclareMathOperator{\dom}{dom}

\usepackage{enumitem} % Customize lists
\setlist[enumerate,1]{label=(\arabic*)}          % First level: 1., 2., 3....
\setlist[enumerate,2]{label=(\arabic{enumi}\alph*)} % Second level: 1.1, 1.2...
\setlist[enumerate,3]{label=(\arabic{enumi}\alph{enumii}\roman{enumiii})} % Third level: 1.1.1, 1.1.2...
\usepackage{subcaption} % use subfigure

\begin{document}
	
\title{Isolated d.c.e.\ degrees and \(\Sigma_1\) induction}

\author[Yiqun Liu]{Yiqun Liu}
\address[Yiqun Liu]{Office of the President\\
National University of Singapore\\
Singapore 119077\\
SINGAPORE}
\email{\href{mailto:liuyq@nus.edu.sg}{liuyq@nus.edu.sg}}

\author[Yong Liu]{Yong Liu}
\address[Yong Liu]{School of Information Engineering\\
Nanjing Xiaozhuang University\\
CHINA}
\email{\href{mailto:liuyong@njxzc.edu.cn}{liuyong@njxzc.edu.cn}}

\author[Peng]{Cheng Peng}
\address[Peng]{Institute of Mathematics\\
Hebei University of Technology\\
CHINA}
\email{\href{mailto:pengcheng@hebut.edu.cn}{pengcheng@hebut.edu.cn}}
	


\thanks{Peng's research was partially funded by the Science and
Technology Project of the Hebei Education Department (No.~QN2023009) and the National Natural Science Foundation of China (No.~12271264). Yong Liu's research was partially funded by Nanjing Xiaozhuang University (No.~2022NXY39).}
	
\begin{abstract}
A Turing degree is d.c.e.\ if it contains a set that is the difference of two c.e.~sets. A d.c.e.\ degree \(\bm{d}\) is \emph{isolated} by a c.e.\ degree \(\bm{a}<\bm{d}\) if all c.e.\ degrees that are below \(\bm{d}\) are also below \(\bm{a}\); \(\bm{d}\) is \emph{isolated from above} by a c.e.\ degree \(\bm{a}>\bm{d}\) if all c.e.\ degrees that are above \(\bm{d}\) are also above \(\bm{a}\). In this paper, we study the inductive strength of both isolated and upper isolated d.c.e.\ degrees from the point of view of reverse recursion theory. We show that
(1) \(P^{-} + B\Sigma_1 + \text{Exp} \vdash I\Sigma_1 \leftrightarrow\) There is an isolated proper d.c.e.\ degree below \(\bm{0}'\);
(2) \(P^{-} + B\Sigma_1 + \text{Exp} \vdash  I\Sigma_1 \leftrightarrow\) There is an upper isolated proper d.c.e.\ degree below \(\bm{0}'\).
\end{abstract}
	
\maketitle

\section{Introduction}

Reverse recursion theory, inspired by the program of reverse mathematics, seeks to clarify the proof-theoretic strength required to establish theorems in computability theory. Its central aim is to determine precisely which levels of induction are necessary and sufficient to prove classical recursion-theoretic results. 
% In particular, it analyzes the induction strength needed for such theorems and, where possible, identifies exact equivalences between induction schemes and recursion-theoretic statements.

Paris and Kirby~\cite{pk1978} introduced a hierarchy of fragments of Peano arithmetic (\(\mathsf{PA}\)), showing that \(I\Sigma_n\) (the induction scheme for \(\Sigma_n\) formulas) is strictly stronger than \(B\Sigma_n\) (the bounding scheme for \(\Sigma_n\) formulas), and that \(B\Sigma_{n+1}\) is strictly stronger than \(I\Sigma_n\), all over the base theory \(P^- + I\Sigma_0 + \text{Exp}\). Here, \(P^-\) denotes \(\mathsf{PA}\) without the induction scheme, and Exp asserts the totality of the exponential function \(x \mapsto 2^x\), which is provable in \(P^- + I\Sigma_1\). All standard notions from classical computability theory (such as c.e.~sets and Turing reducibility) can be formalized within \(P^- + B\Sigma_1 + \text{Exp}\).

A central focus in reverse recursion theory is analyzing the induction strength required for priority constructions. Of particular interest are finite injury priority methods, which can be classified into two types: bounded and unbounded. 
Bounded type constructions have an \emph{effective bound} on the number of times each requirement can be injured---examples include the Friedberg-Muchnik theorem and the existence of a noncomputable low c.e.\ degree. Unbounded type constructions lack such effective bounds, as seen in the Sacks Splitting Theorem.

Bounded constructions can typically be carried out in \(I\Sigma_1\) relatively straightforwardly. More sophisticated techniques allow proving the Friedberg-Muchnik theorem in \(B\Sigma_1\)~\cite{cm1992}. In contrast, Slaman and Woodin~\cite{sw1989} demonstrated that \(B\Sigma_1\) is insufficient for the Sacks Splitting Theorem. In fact, using Shore's blocking methods, Mytilinaios~\cite{Mytilinaios1989} showed that Sacks splitting theorem is equivalent to \(I\Sigma_1\) over \(P^-+B\Sigma_1+Exp\). Even a bounded type construction may be equivalent to \(I\Sigma_1\) over \(P^-+B\Sigma_1+Exp\), for example, the existence of a noncomputable low c.e.\ degree~\cite{cy2007}.

For further introduction to the subject of reverse recursion theory, we refer the reader to the excellent survey paper by Chong, Li, and Yang~\cite{cly2014bsl}.

In classical computability theory, a set is called d.c.e.\ if it is the difference of two c.e.~sets, and a Turing degree is d.c.e.\ if it contains a d.c.e.\ set. 
A d.c.e.\ degree is \emph{proper} if it contains a d.c.e.\ set but contains no c.e.\ sets. Cooper~\cite{cooper} showed the existence of a proper d.c.e.\ degree. 
The d.c.e.\ degrees form a structure different from the c.e.~degrees. One fundamental difference, shown by Cooper, Harrington, Lachlan, Lempp, and Soare~\cite{chlls}, is that there is a maximal incomplete d.c.e.\ degree, i.e., a d.c.e.\ degree \(\bm{d} < \bm{0}'\) such that there is no d.c.e.\ degree \(\bm{c}\) with \(\bm{d} < \bm{c} < \bm{0}'\). 
Thus, upward density fails in the d.c.e.\ degrees; however, downward density holds in the d.c.e.\ degrees by an observation of Lachlan. 
Though the d.c.e.\ degrees fail to be dense, Cooper and Yi~\cite{cooperyi1995} showed that the following weak density holds in the d.c.e.\ degrees: for any c.e.~degree \(\bm{a}\) and d.c.e.\ degree \(\bm{d}\) with \(\bm{a}<\bm{d}\), there is a d.c.e.\ degree \(\bm{c}\) such that \(\bm{a}<\bm{c}<\bm{d}\). 
They also showed that the requirement ``\(\bm{c}\) d.c.e.\ '' is necessary, by showing that there exists an \emph{isolated} d.c.e.\ degree~\cite{cooperyi1995}.

\begin{definition}
A d.c.e.\ degree \(\bm{d}\) is \emph{isolated} by a c.e.~degree \(\bm{a}\) if \(\bm{a} < \bm{d}\) and every c.e.~degree below \(\bm{d}\) is also below \(\bm{a}\). In this case, \(\bm{d}\) and \(\bm{a}\) form an \emph{isolation pair}, and \(\bm{d}\) is called an \emph{isolated degree}.
\end{definition} 

Ishmukhametov and Wu~\cite{iw2002} showed that the degrees in an isolation pair need not be close to each other: there is an isolation pair \(\bm{d}\) and \(\bm{a}\) such that \(\bm{a}\) is low and \(\bm{d}\) is high. 
Efremov~\cite{efremov1998} and Wu~\cite{wu2004} consider a similar notion of isolation pair in the following way.

\begin{definition}
A d.c.e.\ degree \(\bm{d}\) is \emph{isolated from above} by a c.e.~degree \(\bm{a}\) if \(\bm{d} < \bm{a}\) and every c.e.~degree above \(\bm{d}\) is also above \(\bm{a}\). In this case, \(\bm{d}\) and \(\bm{a}\) form an upper isolation pair, and \(\bm{d}\) is called an \emph{upper isolated degree}.
\end{definition}

A maximal d.c.e.\ degree \(\bm{d}\) is isolated from above by \(\bm{0}'\). Efremov~\cite{efremov1998} proved, among other results, that there exists a low d.c.e.\ degree \(\bm{d}\) which is isolated from above by \(\bm{0}'\); however, Downey and Yu~\cite{dy2004} showed that such a \(\bm{d}\) cannot be both low and maximal. 
Wu~\cite{wu2004} proved the existence of a \emph{bi-isolated degree} (one that is both isolated and upper isolated). Liu~\cite{liu2019} later improved this result by showing that there exists an isolated maximal d.c.e.\ degree.

From the perspective of reverse recursion theory, d.c.e.\ degrees become very subtle objects. In an \(I\Sigma_1\) model, they behave as one would expect; Kontostathis~\cite{kon93} observed that Cooper's proof of the existence of a proper d.c.e.\ degree can be carried out in \(I\Sigma_1\). However, in \(B\Sigma_1\) models (without \(I\Sigma_1\)), maintaining the transitivity of \(\le_T\) requires adopting a modified definition of Turing reducibility (see Definition~\ref{def:turing_reducibility}). Li~\cite{li2014} constructs a (proper) d.c.e.\ degree that does not lie below \(\bm{0}'\) in such settings. The same paper also concludes that, over the base theory \(P^{-} + B\Sigma_1 + \text{Exp}\), \(I\Sigma_1\) is equivalent to the existence of a proper d.c.e.\ degree below \(\bm{0}'\). 

In this paper, we study the existence of isolated d.c.e.\ degrees and upper isolated d.c.e.\ degrees in models of \(I\Sigma_1\). Note that an isolated d.c.e.\ degree is automatically proper because the Sacks Splitting Theorem holds in models of \(I\Sigma_1\)~\cite{Mytilinaios1989} (see Lemma~\ref{lem:isolate_proper}). However, we do not know whether an upper isolated d.c.e.\ degree is automatically proper, as there is no similar theorem known in \(I\Sigma_1\). Our main theorems are as follows.

\begin{theorem}\label{theorem:isolated}
\(P^{-} + I\Sigma_1 \vdash\) There is an isolated proper d.c.e.\ degree below \(\bm{0}'\).
\end{theorem}

\begin{theorem}\label{theorem:upperisolated}
\(P^{-} + I\Sigma_1 \vdash\) There is an upper isolated proper d.c.e.\ degree below \(\bm{0}'\).
\end{theorem}

As \(I\Sigma_1\) is equivalent to the existence of a proper d.c.e.\ degree below \(\bm{0}'\) over the base theory \(P^{-} + B\Sigma_1 + \text{Exp}\)~\cite{li2014}, our theorems can be rephrased as 
\begin{corollary}\label{cor:mainresult}
\begin{enumerate}
    \item \(P^{-} + B\Sigma_1 + \text{Exp} \vdash I\Sigma_1 \leftrightarrow\) There exists an isolated proper d.c.e.\ degree below \(\bm{0}'\).
    \item \(P^{-} + B\Sigma_1 + \text{Exp} \vdash I\Sigma_1 \leftrightarrow\) There exists an upper isolated proper d.c.e.\ degree below \(\bm{0}'\).\qed{}
\end{enumerate}
\end{corollary}

The classical constructions for both isolated pairs and upper isolated pairs are presented as infinite injury priority methods. In this paper, we present finite injury constructions of the ``bounded'' type for both isolated pairs and upper isolated pairs. The construction for Theorem~\ref{theorem:isolated} is achieved by carefully analyzing the conflicts between requirements, while the construction for Theorem~\ref{theorem:upperisolated} is inspired by certain aspects of Shore's work~\cite{shore} on the priority method.

This paper is organized as follows. In Section~\ref{sec:preliminary}, we recall some notions and theorems that will be used in this paper. In Section~\ref{sec:isolated}, we prove Theorem~\ref{theorem:isolated}. In Section~\ref{sec:upperisolated}, we prove Theorem~\ref{theorem:upperisolated}. Finally, in Section~\ref{sec:oq}, we conclude with some further open questions.

\section{Preliminaries}\label{sec:preliminary}

\subsection{The inductive hierarchy}

We work in the language of first-order arithmetic \(\mathcal{L}(0, 1, +, \times, \text{exp}, <)\), where exp denotes the exponential function \(x\mapsto 2^x\). Here, \(P^-\) denotes the axioms of Peano arithmetic without induction, and Exp denotes the axiom \(\forall x \exists y (y = \text{exp}(x))\). \(\Sigma_0=\Pi_0\) formulas are those that contain only bounded quantifiers. A \(\Sigma_{n+1}\) formula has the form \(\exists \vec{x} \varphi\) for some \(\Pi_n\) formula \(\varphi\), and a \(\Pi_{n+1}\) formula has the form \(\forall \vec{x} \psi\) for some \(\Sigma_n\) formula \(\psi\). A \(\Delta_n\) formula is a \(\Sigma_n\) formula that is provably equivalent to some \(\Pi_n\) formula over a base theory.
% A \(\Sigma_n\) formula has the form \(\exists \vec{x}_1 \forall \vec{x}_2 \dots \phi\) with \(n\) alternating quantifiers, where \(\phi\) contains only bounded quantifiers. A $\Pi_
% \(\Pi_n\)-formulas and \(\Delta_n\)-formulas are defined in the standard way.

If \(\Gamma\) is a set of formulas, \(I\Gamma\) is the \emph{induction scheme} for elements of \(\Gamma\). It consists the universal closure of the formulas
\[
(\theta (0) \wedge \forall x [\theta(x) \rightarrow \theta(x+1)]) \rightarrow \forall x\theta(x)
\]
for \(\theta\) in \(\Gamma\). \(B\Gamma\) is the \emph{bounding scheme} for elements of \(\Gamma\). It consists the universal closure of the formulas 
\[
\forall y(\forall x < y\exists z \theta(x,z) \rightarrow \exists t\forall x < y\exists z < t\theta(x,z))
\]
for \(\theta\) in \(\Gamma\). 

\begin{theorem}
(Kirby and Paris~\cite{kirby1992}) Work in \(PA^-+I\Sigma_0\). For all \(n \geq 1\),
\begin{enumerate}
    \item \(I\Sigma_n \Leftrightarrow I\Pi_n\)
    
    \item \(B\Sigma_{n+1} \Leftrightarrow B\Pi_n\)
    
    \item \(I\Sigma_{n+1} \Rightarrow B\Sigma_{n+1} \Rightarrow I\Sigma_n\), and the implications are strict.
\end{enumerate}
(Slaman~\cite{slaman2004}) Work in \(P^-+I\Sigma_0+Exp\). For all \(n\ge 1\),
\begin{enumerate}[resume]
    \item \(B\Sigma_n\Leftrightarrow I\Delta_n\).
\end{enumerate}
\end{theorem}
These results demonstrate the inductive hierarchy.

% A model of \(P^-+I\Sigma_0+Exp\) is said to be a \(B\Sigma_n\) model if it satisfies \(B\Sigma_n\) but not \(I\Sigma_n\).

\subsection{Finiteness}
Let \(\mathcal{M}\) be a model of \(P^- + I\Sigma_0 + \text{Exp}\), and let \(M\) be the domain of \(\mathcal{M}\). When \(\mathcal{M}\) is nonstandard, finiteness is defined in terms of codability.
\begin{definition}
\begin{enumerate}
    \item A set \(A\subseteq M\) is \textbf{bounded} if there exists some \(b \in M\) such that \(\forall x\in A~(x< b)\).
    \item A set \(A\subseteq M\) is \(\mathcal{M}\)\textbf{-finite} if it is coded by some \(c\in M\), i.e., for any \(i\), \(i\in A\) if and only if the \(i\)-th digit in the binary expansion of \(c\) is 1.
\end{enumerate}
\end{definition}

\(\mathcal{M}\)-finite sets are bounded in \(\mathcal{M}\), but the converse is false since \(\omega\) is bounded but not \(\mathcal{M}\)-finite in nonstandard models.
The following lemma is well-known.
\begin{lemma}[H. Friedman]\label{lem:Friedman}
Let \(n \geq 1\) and \(\mathcal{M} \models P^- + I\Sigma_n\). Then every bounded \(\Sigma_n\) set is \(\mathcal{M}\)-finite, and every partial \(\Sigma_n\) function maps a bounded set to a bounded set.
\end{lemma}

\subsection{Turing reducibility}\label{sec:turing}
Let \(\mathcal{M}\) be a model of \(P^- + I\Sigma_0 + \text{Exp}\). We usually identify an \(\mathcal{M}\)-finite set with its code in \(\mathcal{M}\) (e.g., in the following definition of Turing functional and Turing reducibility involving quadruples). 
\begin{definition}\label{def:turing_reducibility}
 The basic recursion theoretic notions are the following.
\begin{enumerate}
\item A set \(A\subseteq M\) is \textbf{c.e.}~if it is \(\Sigma_1\) definable (with parameters from \(M\)) in \(\mathcal{M}\). 

\item A set \(A\subseteq M\) is \textbf{computable} if \(A\) and its complement \(\bar{A}\) are both c.e. 

\item A \textbf{Turing functional} is a c.e.~set \(\Phi\) of quadruples \(\langle X, y, P, N \rangle\), where \(X, P, N\) are \(\mathcal{M}\)-finite sets, \(y\) is a number, and
\begin{enumerate}
\item[(i)] \(\langle X, y, P, N \rangle \in \Phi \rightarrow P \cap N =\emptyset\),
\item[(ii)] \(( \langle X, y, P, N \rangle \in \Phi \wedge P' \cap N' = \emptyset \wedge P \subseteq P' \wedge N \subseteq N' \wedge X' \subseteq X) \rightarrow \langle X', y, P', N' \rangle \in \Phi\), 
\item[(iii)] \((\langle X, y, P, N \rangle , \langle X, y', P, N \rangle \in \Phi) \rightarrow y = y'\).
\end{enumerate}

\item Let \(A, B \subseteq M\). \(A\) is \textbf{Turing reducible} to \(B\) (or ``\(A\) is \textbf{computable in} \(B\)''), denoted by \(A \leq_T B\), if there is a Turing functional \(\Phi\) such that for any \(\mathcal{M}\)-finite set \(X\), 
\begin{enumerate}
\item[(i)] \(X \subseteq A \Longleftrightarrow  \exists P \subseteq B \exists N \subseteq \bar{B} \langle X, 1, P, N \rangle \in \Phi\),
\item[(ii)] \(X \subseteq \bar{A} \Longleftrightarrow \exists P \subseteq B \exists N \subseteq \bar{B}  \langle X, 0, P, N \rangle \in \Phi\).
\end{enumerate}
\end{enumerate}
\end{definition}

\begin{remark}
(1) For a c.e.~set \(A\), let \(A_s \subseteq A\) be the collection of numbers enumerated into \(A\) by stage \(s\). If \(\mathcal{M}\) satisfies \(P^- + I\Sigma_1\), then, by Lemma~\ref{lem:Friedman}, \(A_s\) is \(\mathcal{M}\)-finite for any \(s\). 
(2) If \(A\leq_T B\) via \(\Phi\), then we denote it as \(A = \Phi^B\).
\end{remark}

Intuitively, for a Turing functional \(\Phi\), \(\langle X, y, P, N \rangle \in \Phi\) means the program \(\Phi\) with input \(X\) produces output \(y\), whenever \(P\) is some positive part of the oracle and \(N\) is some negative part of the oracle.
The definition of Turing reducibility is a setwise reduction, that is, the reduction procedure is designed to answer questions about \(\mathcal{M}\)-finite sets \(X\), rather than individual numbers \(x\), with the help of an oracle. The reason is that the pointwise Turing reduction (which is obtained by substituting a number ``\(x\)'' for an \(\mathcal{M}\)-finite set ``\(X\)'' in the definition) is not transitive.
%It is straightforward to verify that \(\leq_T\) is transitive.

In the rest of the paper, we write the construction in a classical way, i.e., with respect to the pointwise reducibility. Whenever we show that one set \(A\)
computes another \(B\), we produce a reduction of \(B\) to \(A\) with an especially strong property, that is, initial segments of \(B\) are reduced to initial segments of \(A\).
%More precisely, suppose we want to construct a Turing functional \(\Gamma\) such that \(\Gamma^A = B\), when we enumerate a quadruple to define \(\Gamma^A(x) = B(x)[s]\) for some \(x\) at stage \(s\), we actually mean that we enumerate quadruples \(\langle  B_s \cap [0, x], 1, P, N \rangle\) and \(\langle \overline{B}_s \cap [0, x], 0, P, N \rangle\) into \(\Gamma\).
Thus, in our construction, the Turing functional \(\Gamma\) is defined on an initial segment of \(B\) at any stage \(s\); this ensures that \(\Gamma^A = B\) in the sense of setwise reducibility.

% Since Exp holds in \(\mathcal{M}\), we can apply G\"{o}del's coding method to obtain an effective enumeration of all c.e.\ sets. We fix \(\{W_e: e\in M\}\) to be an effective enumeration of all c.e.\ sets in \(\mathcal{M}\). Note that we can uniformly modify the enumeration of \(W_e\) to produce a Turing functional \(\Phi_e\) in a computable way; we also fix an effective enumeration of all Turing functionals \(\{\Phi_e: e\in M\}\).
% Define \(\emptyset '=\{e: e\in W_e\}\), \(\text{deg}(A)=\{B: B\equiv_T A\}\), \(\bm{0}=\text{deg}(\emptyset)\) and \(\bm{0'}=\text{deg}(\emptyset ')\) as usual. Also, we use \(\{W_e^A: e\in M\}\) to denote a fixed enumeration of all \(\Sigma_1(A)\)-subsets of \(\mathcal{M}\). Define \(A' = \{e : e\in W^A_e \}\) and say \(A\) is low if \(A'\leq_T \emptyset'\).

% \subsection{d.c.e.\ set in a model}
% A set \(D\subseteq M\) is \emph{d.c.e.} if it is the difference of two c.e.~sets.
% As mentioned in the introduction, we have the following: 
% \begin{lemma}[Li~\cite{li2014}]\label{lem:li2014}
% (1) In any \(B\Sigma_1\) model, every d.c.e.\ degree below \(\bm{0}'\) is c.e.\\
% (2) \(P^-+ I\Sigma_1\vdash\) Every d.c.e.\ set is Turing reducible to \(\emptyset '\).
% \end{lemma}
% The above lemma shows that Theorem~\ref{theorem:isolated} and Theorem~\ref{theorem:upperisolated} imply Corollary~\ref{cor:mainresult}. 
% We remark that the conclusion of (2) fails if there is no condition \(I\Sigma_1\).

%D_s not finite
\section{Proof of Theorem~\ref{theorem:isolated}}\label{sec:isolated}
We will construct an isolated d.c.e.\ degree using a finite injury argument of the ``bounded'' type. We can omit explicit requirements for properness and being below \(\bm{0}'\) due to the following lemmas:

\begin{lemma}\label{lem:isolate_proper}
    Over \(P^-+I\Sigma_1\), every isolated d.c.e.\ degree is proper.
\end{lemma}
\begin{proof}
    Let \(\bm{d}\) be a d.c.e.\ degree isolated by \(\bm{a}=\deg(A)\), where \(A\) is c.e.\ Suppose for contradiction that \(\bm{d}\) is c.e.\@, witnessed by some c.e.\ set \(D\). By the Sacks Splitting Theorem in \(I\Sigma_1\) models~\cite{Mytilinaios1989}, we can split \(D\) into c.e.\ sets \(D_0\) and \(D_1\) such that
    \[
        D \equiv_T D_0 \oplus D_1 \leq_T A <_T D.
    \]
    This yields the contradiction \(D <_T D\).
\end{proof}

The following lemma is also well-known (see Li~\cite{li2014} for reference).
\begin{lemma}\label{lem:below_zeroprime}
    Over \(P^-+I\Sigma_1\), every d.c.e.\ degree is below \(\bm{0}'\). \qed{}
\end{lemma}

\subsection{Requirements}
Let \(D\) be a d.c.e.\ set, the \emph{Lachlan set} of \(D\) is defined by
\[
L(D)=\{s: \exists z~[ z \in D_s - D_{s-1} \text{~and~} \exists t>s (z \notin D_t) ] \}
\]
where \(D_s\) is a fixed computable approximation of \(D\). It is easy to see that \(L(D)\) is c.e.~and \(L(D)\leq_T D\). We will build a d.c.e~set \(D\) satisfying the following requirements (where \(\Phi\), \(\Psi\) denote Turing functionals and \(W\) denotes c.e.~set):
\begin{itemize}
\item [\(P_e\):] \(\Phi_e^{L(D)} \neq D\)
\item [\(N_e\):] \(\Psi_{e}^D = W_e  \rightarrow \exists \Gamma ( \Gamma^{L(D)} = W_e ) \)
\end{itemize}
Then \(D\) and \(L(D)\) form an isolated pair by the requirements.\footnote{Alternatively, one can replace the \(P_e\)-requirement with \(\Phi_e^{A} \neq D\), the \(N_e\)-requirement with \(\Psi_{e}^{A\oplus D} = W_e  \rightarrow \exists \Gamma ( \Gamma^{A} = W_e ) \) and construct the d.c.e.\ set \(D\) and c.e.~set \(A\).}

\subsection{\texorpdfstring{P-strategies}{P-strategies}} The \(P\)-strategy is the Friedberg-Muchnik strategy, it picks a witness \(z\) and waits for \(\Phi^{L(D)} (z)\downarrow = 0\), if so, then it puts \(z\) into \(D\) and sets up a restraint on \({L(D)}\) to make the diagonalization \(\Phi^{L(D)}(z)\downarrow\neq D(z)\) succeeds.

\subsection{Disagreement argument}
The basic idea for the \(N\)-strategy is called the disagreement argument (e.g.~such argument can be found in~\cite{iw2002}). For an \(N_e\)-requirement, define the length of agreement function
\[
l(e, s) = \max\{y: \forall x < y (\Psi_e^D (x)[s]\downarrow = W_{e,s}(x))\}
\]
say \(s\) is an \(e\)-expansionary stage if \(l(e, s)  > l(e, t)\) for all \(t<s\). An \(N\)-strategy would like to define and maintain \(\Gamma^{L(D)}=W\) at its expansionary stages (we omit the subscript \(e\) for brevity), it starts with waiting for an expansionary stage, if it finds an expansionary stage \(s^*\), then it defines \(\Gamma^{L(D)}(x)[s^*]=W_{s^*}(x)\) for those undefined \(x\leq s^*\). Assume the \(N\)-strategy does not set up a \(D\)-restraint, then a \(P\)-strategy may put some element \(z\) into \(D\) to injure the computation \(\Psi^D (x)[s^*]\), this provides chances for \(x\) to enter \(W\) after stage \(s^*\) (under the assumption of \(e\)-expansionary stage), which makes \(\Gamma^{L(D)} (x)\) incorrect. The strategy is to recover the old \(\Psi^D(x)\)-computation by extract \(z\) out of \(D\), and such actions would make a disagreement \(\Psi^D(x)\neq W(x)\) (we call it receives disagreement). Since \(W\) is c.e.~and lower priority requirements can not injure this recovered computation \(\Psi^D (x)\) by initialization, the inequality is satisfied forever.
Moreover, let \(s_z\) denote the stage when \(z\) enters \(D\), then whenever the \(N_e\)-strategy extract \(z\) out of \(D\), it will put \(s_z\) into \(L(D)\) (by the definition of Lachlan set). This ensures that the \(N_e\)-strategy will not injure higher priority \(N\)-strategies: say \(N_{i}\) is higher than \(N_e\), suppose that \(N_e\) extract some element \(z\) out of \(D\), and this would affect some \(N_{i}\)-computation \(\Psi_{i}^D (n)\) and make \(\Gamma_{i}^{L(D)} (n)\) incorrect (due to \(n\) entners \(W\)). But now the enumeration of \(s_z\) into \(L(D)\) allows \(N_{i}\) to redefine \(\Gamma_{i}^{L(D)} (n)\), \(N_e\)'s action does not injure \(N_{i}\).

In the usual isolated pair construction, the strategy is organized in a way that the \(N\)-node has different outcomes to distinguish whether or not there are infinitely many expansionary stages, so we need \(I\Sigma_2\) to perform such a construction. In spite of that difficulty we can adopt the following strategy.

\subsection{\texorpdfstring{N-strategies}{N-strategies}}
Our strategy is a modification of the classical \(N\)-strategy by carefully analyzing the conflicts of requirements. For convenience, we define the number of \(D(z)\)-changes at stage \(s\) as
\[
D^{\#}_s (z)=|\{t\leq s: D_t(z)\neq D_{t-1}(z)\}|
\]
then \(D^{\#}_s (z)\leq 2\) if \(D\) is d.c.e. We first give some notations: At stage \(s\), we say \(D\) \emph{is restorable to} \(\sigma\in 2^{<\omega}\) if \(\forall z (D_s(z)\neq \sigma(z)\rightarrow D^{\#}_s (z)=1)\), that is, the changes between \(D_s\) and \(\sigma\) only due to enumeration but not extraction of some elements; and say \emph{an old \(\Psi_e^D(n)\)-computation is restorable} if \(D\) is restorable to some  \(D_{s_0}\upharpoonright\psi_e(n)[s_0]\) such that \(s_0 \leq s\) and \(\Psi_e^{D}(n)[s_0]\downarrow=0\), that is, we can restore an old \(\Psi_e^D(n)\)-computation with value \(0\) by extracting some elements out of \(D\). Note that if an old \(\Psi_e^D(n)\)-computation is restorable, then there could be more than one old computations one can restore.
%(say another \(D_{s_1}\upharpoonright\psi_e(n)[s_1]\) with \(s_0<s_1 \leq s\) and \(\Psi_e^{D}(n)[s_1]\downarrow=0\)).

An \(N_e\)-strategy acts at stage \(s\) as follows:
First check if \(N_e\) has received disagreement (it will be defined later) or \(s\) is an \(e\)-expansionary stage. If \(N_e\) has received disagreement or \(s\) is not an \(e\)-expansionary stage, then \(N_e\) succeeds. Otherwise, let \(n\) be the least such that \(\Gamma_e^{L(D)}(n)\) is undefined or incorrect:

(a) \(\Gamma_e^{L(D)}(n)\) is undefined. There are two subcases: 

(a1) If an old \(\Psi_e^D(n)\)-computation is restorable, then we define \(\Gamma_e^{L(D)}(n)=W(n)\) with the same use.

(a2) Otherwise, we define \(\Gamma_e^{L(D)}(n)=W(n)\) with a fresh use.

(b) \(\Gamma_e^{L(D)}(n)\) is incorrect. So \(n\) enters \(W\) at stage \(s\). In this case, we may verify that an old \(\Psi_e^D(n)\)-computation must be restorable, we restore \(D\) to the oldest one. Moreover, if we extract any \(z\) out of \(D\), we put \(s_z\) into \(L(D)\) (recall that \(s_z\) is the stage when \(z\) enters \(D\)). We say \(N_e\) receives disagreement and initialize the lower priority strategies. 

We remark that in case (a1), we have \(\gamma_e(n)[s]=\gamma_e(n)[s-1]\) (it is possible that \(\gamma_e(n)[s]<\psi_e(n)[s]\)), and we will see that if \(\Psi_e^D(n)\) converges and correct, then \(\Gamma_e^{L(D)}(n)\) is defined and correct (Lemma~\ref{lem:Nsatisfied}). In case (a2), we have \(\gamma_e(n)[s]\geq\psi_e(n)[s]\), and in case (b), if there are more than one old computations one can restore, then we have to restore it to oldest one (the formal meaning will be described in the Construction part), thus in both case, we have that if the \(\Psi_e^D(n)\)-computation is injured later and can not be restored (since we extract some \(z\) from \(D\)), then we can redefine the corresponding \(\Gamma_e^{L(D)}(n)\) (by putting \(s_z\) into \(L(D)\)). It follows that the \(N_e\)-requirement is satisfied. Moreover, the \(N_e\)-requirement in isolation only injure other requirements once (when it receives disagreement).

\subsection{Construction}
We build a d.c.e.\ set \(D\) meeting the requirements:

\(P_e: \Phi_e^{L(D)} \neq D\) and \(N_{\langle e_0, e_1\rangle}: \Psi_{e_0}^D = W_{e_1}  \rightarrow \exists \Gamma ( \Gamma^{L(D)} = W_{e_1})\).

Here \(L(D)\) is the Lachlan set of \(D\), so whenever we extract any \(z\) out of \(D\), we will put \(s_z\) into \(L(D)\). Let \(e=\langle e_0, e_1\rangle\), the priority order is \(P_0>N_0>P_1>N_1>P_2>N_2>\dots\). The construction is done in stages \(s\geq 0\) and each stage \(s\) is divided into substages \(t<s\). In stage \(s\) and at substage \(t<s\), we consider \(P_e\) if \(t=2e\) and \(N_e\) if \(t=2e+1\) as follows, then we go to the next substage, when we reach substage \(s\), we go to next stage \(s+1\).

\textit{Case 1}. \(t=2e\) for the \(P_e\)-requirement.
 
(1) If it receives disagreement, then do nothing.

(2) If the witness \(x_{P_e}\) is undefined, then pick \(x_{P_e}\) as fresh. 

(3) Otherwise, we check if \(\Phi_e^{L(D)} (x_{P_e}) [s] \downarrow = 0\). If yes, then we put \(x_{P_e}\) into \(D\), say \(P_e\) receives disagreement and initialize the lower priority strategies. If no, then do nothing.  

\textit{Case 2}. \(t=2e+1\) for the \(N_e\)-requirement.

We will have for each \(n\), if \(\Gamma_e^{L(D)}(n)[s]\) is defined, then there is a so called \emph{disagreement pair} \((n;\sigma,\tau)\) defined during the construction.

(1) If it receives disagreement, then do nothing.

(2) If  \(s\) is not an \(e\)-expansionary stage, then do nothing.

(3) Otherwise, let \(n\) be the least such that \(\Gamma_e^{L(D)}(n)[s]\) is undefined or incorrect.

(a) \(\Gamma_e^{L(D)}(n)[s]\) is undefined.
\begin{itemize}
    \item[(a1)] if there is an \((n;\sigma,\tau)\) for some \(\sigma\) and \(\tau\) has been defined such that \(D_s\) is restorable to \(\sigma\), then we define \(\Gamma_e^{L(D)}(n)[s]=W_{e_1,s}(n)\) with the same use as \(k=|\tau|\) and define \((n;\sigma,{L(D)}_s\upharpoonright k)\) as the disagreement pair (the old one is canceled).
    \item[(a2)] otherwise, we define \(\Gamma_e^{L(D)}(n)[s]=W_{e_1,s}(n)\) with a fresh use \(k\) and define \((n;D_s\upharpoonright \psi_{e_0}(n)[s],{L(D)}_s\upharpoonright k)\) as the disagreement pair.
\end{itemize}

(b) \(\Gamma_e^{L(D)}(n)[s]\) is incorrect. Let \((n;\sigma,\tau)\) be the disagreement pair such that \(\tau\subseteq {L(D)}_s\). We restore \(D_s\) to \(\sigma\), say \(N_e\) receives disagreement and initialize the lower priority strategies.

\subsection{Verification}
For the \(N_e\)-strategy, we have the following lemmas.
\begin{lemma}
In the case (b) (incorrect) of the Construction, \(D_s\) is restorable to \(\sigma\).
\end{lemma}
\begin{proof}
Let \(z\) be any number such that \(D_s(z)\ne \sigma(z)\). If \(D^\#_s(z)=2\), then \({L(D)}^\#_s(s_z)=1\) and \(s_z\leq |\tau|\) by the case (a). Then \(\tau\nsubseteq {L(D)}_s\), we shall have case (a) of undefined. So \(D^\#_s(z)=1\), \(D_s\) is restorable to \(\sigma\).
\end{proof}

Although we do not distinguish whether or not there are infinitely many expansionary stages as in the traditional way, we have the following property.

\begin{lemma}\label{lem:Nsatisfied}
If \(\Psi_{e_0}^D= W_{e_1}\) holds, then \(\Gamma_e^{L(D)}\) is total and correct.
\end{lemma}
\begin{proof}
If not, let \(n\) be the least such that \(\Gamma_e^{L(D)}(n)\uparrow\) or \(\Gamma_e^{L(D)}(n)\downarrow\ne W_{e_1}(n)\). Since \(\Psi_{e_0}^D(n)\) converges, there exists stage \(s_0\) such that \(D_{s_0}\upharpoonright\psi_{e_0}(n)=D\upharpoonright\psi_{e_0}(n)\). For any \(z>\psi_{e_0}(n)\), if \(D^\#_s(z)=2\) for some \(s>s_0\), then we have \(s_z>\gamma_e(n)[s_0]\) by the case (a), so putting numbers into \(L(D)\) will not injure \(\Gamma_e^{L(D)}(n)[s_0]\), \(\Gamma_e^{L(D)}(n)\) is defined. Since \(\Psi_{e_0}^D= W_{e_1}\), \(N_e\) will not receive disagreement, \(\Gamma_e^{L(D)}(n)\) is correct by the case (b). This contradicts to the hypothesis.
\end{proof}

In the construction, the only action of a \(P_e\)- or \(N_e\)-strategy is when they receive disagreement. Note that by initialization, if some strategy receive disagreement, then the numbers picked by lower priority strategies are too large to injure this strategy.

\begin{theorem}
Each requirement is satisfied.
\end{theorem}
\begin{proof}
Since each strategy only initializes the lower priority strategies at most once, for every strategy \(R\) there is some stage \(s_0\) such that it is not be initialized after stage \(s_0\). 

(i) If \(R\) is the \(P_e\)-strategy, it is straightforward to show that it wins from stage \(s_0\).

(ii) If \(R\) is the \(N_e\)-strategy, the above lemma show that it is not affected by the actions of lower strategies and it wins from stage \(s_0\).
\end{proof}

Thus the construction uses a bounded type finite injury priority method.

\subsection{\texorpdfstring{\(\Sigma_1\)-Induction}{Sigma1-Induction}}\label{subsec:sigma1}
Finally, we moved on to the strength of induction and give the proof of Theorem~\ref{theorem:isolated}.

\newtheorem*{theorem2}{Theorem 1.3}
\begin{theorem2}
\(P^{-} + I\Sigma_1 \vdash\) There is an isolated proper d.c.e.\ degree below \(\bm{0}'\). 
\end{theorem2}
\begin{proof}
Suppose model \(\mathcal{M}\models P^- + I\Sigma_1\) and we implement the above construction in \(\mathcal{M}\). Let \({\{R_e\}}_{e\in M}\) be all the requirements, we show that each \(R_e\) is satisfied.

By \(\Pi_1\)-induction, each \(R_e\) can be injured at most \(2^e -1\) many times.
Let \(A_e=\{\langle i, n\rangle: i<e\) \(\wedge\) \(R_i \text{ acts at least \(n\) times}\}\), then \(A_e\) is \(\Sigma_1\)-definable and bounded. Define the partial computable function \(f: A_e \rightarrow M\) by \(f(\langle i, n\rangle)=s\) where \(R_i\) acts \(n\) times by the end of stage \(s\), then \(f\) is \(\Sigma_1\)-definable. By \(I\Sigma_1\), the range of \(f\) is bounded since \(\text{dom}(f)=A_e\) is bounded. So there exists \(s_0\) such that \(\text{ran}(f)\subseteq [0, s_0)\), no \(R_i\) with \(i<e\) acts at any stage after \(s_0\), \(R_e\) is permanently satisfied at some stage after \(s_0\).
\end{proof}

\section{Proof of Theorem~\ref{theorem:upperisolated}}\label{sec:upperisolated}
The construction employs a bounded finite injury priority argument. While we omit explicit discussion of the requirements ensuring the properness of the d.c.e.\ degrees, these can be incorporated into the construction in a standard manner (observed by Kontostathis~\cite{kon93}). We focus instead on the novel aspects of the construction, particularly the conflicts between \(L\)- and \(R\)-requirements listed below. 
Following Lemma~\ref{lem:below_zeroprime}, we also omit explicit discussion of the requirements ensuring the degree is below \(\bm{0}'\).

\subsection{Requirements and other preliminaries}
Let \(M\) be a model of \(I\Sigma_1\). 
We construct a d.c.e.\ set \(D\) and a computable function \(\Delta\) satisfying, for all \(e\in M\):
\begin{itemize}
    \item[] \(L_e: D'(e)=\lim_s \Delta(e,s)\),
    \item[] \(R_e: D=\Phi_e^{W_e}\to K=\Gamma^{W_e}\).
\end{itemize}
To avoid conflicts with \(L\)\nbd{}strategy, we assume that the halting set \(K\) is a subset of the odd numbers. Note that we can carry out the classical Limit Lemma within \(I\Sigma_1\), so that the \(L_e\)-requirements imply that \(D\) is low.\footnote{See~\cite{cy2007} for some weak version of the limit lemma in the absence of \(I\Sigma_1\).}
We employ a priority tree \(\mathcal{T} = {\{0\}}^{<\omega}\). For each node \(\alpha \in \mathcal{T}\), we designate it as an \emph{\(L_e\)-node} if \(\abs{\alpha} = 2e\), and as an \emph{\(R_e\)-node} if \(\abs{\alpha} = 2e + 1\).
At stage~\(s\), we start at the root of \(\mathcal{T}\). When \(\alpha\) acts, it decides whether \(\alpha\concat 0\) acts or we stop the current stage. If \(\alpha\) acts at stage \(s\), then we say \(s\) is an \(\alpha\)-stage.

Now we describe the strategies for each requirement, assuming no initialization occurs. 

\subsection{\texorpdfstring{\(R_e\)-Strategy in isolation}{Re-strategy in isolation}}\label{sec:R-strategy-in-isolation}
Let \(\alpha\) be the \(R_e\)-node.
The length of agreement for \(\alpha\) is defined for each \(\alpha\)\nbd{}stage \(s\) by 
\[
    \ell(e, s) = \max \{ y: \forall x < y (\Phi_e^{W_e} (x)[s]\downarrow = D_s(x))\}.
\]
An \(\alpha\)\nbd{}stage \(s\) is an \(\alpha\)\nbd{}expansionary stage if \(\ell(e, s)>\ell(e, t)\) for each \(\alpha\)\nbd{}stage \(t<s\). 

The node \(\alpha\) will build a local functional \(\Gamma_\alpha\) such that \(\Gamma_\alpha^{W_e}(x)=K(x)\) for all \(x\in M\), provided that \(\Phi_e^W = D\). 
To ensure the correctness of each \(\Gamma_\alpha^{W_e}(x)=K(x)\), before defining \(\Gamma_\alpha^{W_e}(x)\), we pick a fresh number \(d_{e,x}\), called the \emph{agitator} for \(x\), and wait for an \(\alpha\)\nbd{}expansionary stage \(s\) when \(d_{e,x}<\ell(e,s)\). 
Then we define \(\Gamma^{W_e}(x)=K_s(x)\) with use \(\gamma(x)=\varphi_e^{W_e}(d_{e,x})\). 
Whenever \(0=\Gamma_\alpha^{W_e}(x)\neq K(x) =1\), we enumerate \(d_{e,x}\) into \(D\) to undefine \(\gamma_e^{W_e}(x)\).
%  (see Lemma~\ref{lem:undefined_after_enumerate} below).

The agitator \(d_{e,x}\) is either \emph{undefined}, \emph{defined}, \emph{active}, \emph{enumerated}, or \emph{obsolete}, as described below.

\textbf{\(R_e\)-strategy:} For all \(x\in M\), the agitator \(d_{e,x}\) is initially \emph{undefined}. At stage~\(s\),
\begin{enumerate}
    \item If \(s\) is not an \(\alpha\)-expansionary stage or \(\ell(e,s)<d_{e,x}\) for some \(d_{e,x}\) that is defined, we let \(\alpha\concat 0\) act (and skip the instructions below).
\end{enumerate}
\emph{Below we assume \(s\) is an \(\alpha\)-expansionary stage and \(\ell(e,s)\ge d_{e,x}\) for all \(d_{e,x}\) that are defined.}
\begin{enumerate}[resume]
    \item\label{it:extract} For each \(x\) with \(D_s(d_{e,x})=1\), we extract \(d_{e,x}\) and let \(d_{e,y}\) be \emph{undefined} for \(y\ge x\).
    \item\label{it:enumerate} Let \(x\) be the least (if any) such that \(\Gamma^{W_{e,s}}(x)\downarrow\neq K_s(x)\). \(d_{e,x}\) is \emph{enumerated} into \(D\) so \(D(d_{e,x})=1\), and \(d_{e,y}\) becomes \emph{undefined} for \(y>x\). We stop the current stage.
    \item Let \(y\) be the least (which always exists) such that \(\Gamma^{W_{e,s}}(y)\uparrow\). For each \(x\) with \(y\le x < \ell(e,s)\), 
    \begin{enumerate}
        \item\label{it:easy_case} If \(K_s(x)=1\), we define \(\Gamma^{W_{e,s}}(x)=1\) with use \(\gamma(x)=0\). \(d_{e,x}\) is \emph{obsolete}.
        \item\label{it:get_defined} If \(K_s(x)=0\) but \(d_{e,x}\) is undefined, we let \(d_{e,x}\) be \emph{defined} with a fresh number. % larger than current \(\ell(e,s)\).
        \item\label{it:hard_case} If \(K_s(x)=0\), \(d_{e,x}\) is defined, we define \(\Gamma^{W_{e,s}}(x)=0\) with use \(\gamma(x)=\varphi_e^W(d_{e,x})\). We say \(d_{e,x}\) is \emph{active}.
    \end{enumerate}
    Then we let \(\alpha\concat 0\) act.
\end{enumerate}

If we focus on a particular \(d_{e,x}\) with \(x\) being odd, a typical lifespan of it is as follows:
At first, it is undefined. In Item~\ref{it:get_defined}, \(d_{e,x}\) is \emph{defined} with a fresh number. At the next \(\alpha\)-expansionary stage, it is active in Item~\ref{it:hard_case}. At the next \(\alpha\)-expansionary stage, we have \(x\) enumerated into \(K\) and therefore \(d_{e,x}\) enumerated into \(D\) in Item~\ref{it:enumerate}. Then, at the next \(\alpha\)-expansionary stage, it is extracted and undefined in Item~\ref{it:extract}, and then in Item~\ref{it:easy_case}, it becomes \emph{obsolete}. \(\Gamma^W(x) = K(x)\) is then always correct.

\begin{lemma}\label{lem:R1}
    Suppose \(\Phi_e^W = D\) and \(\Gamma_\alpha^W\) is total. Then we have \(\Gamma_\alpha^W=K\).
\end{lemma}
\begin{proof}
    Given an \(x\), suppose, towards a contradiction, that \(\Gamma_\alpha^{W}(x)\neq K(x)\).
    In such a case, Item~\ref{it:easy_case} does not happen.

    Since \(\Gamma_\alpha^{W_e}(x)\) is eventually defined, let \(s_0\) be the stage when Item~\ref{it:hard_case} happens for the last time. Let \(s_1>s_0\) be the stage when Item~\ref{it:enumerate} happens. Let \(s_2>s_1\) be the next \(\alpha\)-expansionary stage. Then we have 
    \[
        0=D(d_{e,x})[s_0]=\Phi_e^W(d_{e,x})[s_0]\downarrow \neq \Phi_e^W(d_{e,x})[s_2]\downarrow = D(d_{e,x})[s_2]=1
    \]
    which implies \(\varphi_e^W(d_{e,x})[s_0]\neq \varphi_e^W(d_{e,x})[s_2]\). Hence, \(\Gamma_\alpha^{W}(x)\uparrow\) at stage \(s_2\), contradicting the choice of \(s_0\).
\end{proof}

\begin{lemma}\label{lem:R2}
    Suppose \(\Phi_e^W = D\), then \(\Gamma_\alpha^W\) is total.
\end{lemma}

The proof of Lemma~\ref{lem:R2} will be given in the verification section, where we will also address the influence of the \(L_e\)\nbd{}strategy.
With Lemma~\ref{lem:finite_injury}, Lemma~\ref{lem:R2}, and Lemma~\ref{lem:R1}, we have the following lemma:
\begin{lemma}\label{lem:R_satisfied}
    For each \(e\in M\), \(R_e\) is satisfied. \qed{}
\end{lemma}

\subsection{\texorpdfstring{\(\Delta\)-Strategy}{Delta-strategy}}\label{sec:Delta-strategy}
We define \(\Delta(e,s)=0\) for all \(e\ge s\). At the end of stage \(s\), we define 
\[
    \Delta(e,s)=\begin{cases}
        1, & \text{if}\; \Phi_{e,s}^{D_s}(e)\downarrow; \\
        0, & \text{if}\; \Phi_{e,s}^{D_s}(e)\uparrow.
    \end{cases}
\]
The \(L_e\)\nbd{}strategy aims to ensure that \(\lim_s \Delta(e,s)\) exists in \(M\). Note that if this limit exists, it automatically equals \(D'(e)\), as desired.

\subsection{\texorpdfstring{\(L_e\)-Strategy}{Le-strategy}}\label{sec:L-strategy}
Let \(\beta\) be the \(L_e\)-node. In this section, we assume that any \(\alpha\subsetneq \beta\) is not going to be initialized. 
To ensure \(\lim_s \Delta(e,s)\) exists in \(M\), \(\beta\) aims to preserve any computation \(\Phi_{e,s}^{D_s}(e)\) established at stage~\(s\). 
The computation may be injured whenever an \(R_i\)\nbd{}node \(\alpha\subseteq \beta\) (where \(2i=\abs{\alpha}\)) enumerates or extracts its agitator \(d_{i,x}\) (as described in Item~\ref{it:enumerate} and Item~\ref{it:extract} of the \(R\)\nbd{}strategy). 
Our goal is to carefully control the frequency of such injuries.

For an \(R_i\)\nbd{}node \(\alpha\) and an \(L_e\)\nbd{}node \(\beta\) with \(i<e\), \(\beta\) will take control of \(d_{i,2(e-i)}\) and let \(d_{i,j}\) with \(j<2(e-i)\) initialize \(L_e\) at stage \(s\) whenever \(D_s(d_{i,j})\neq D_{s-1}(d_{i,j})\).

For notational convenience, we say \(\sigma\) is a \emph{computation} (for \(L_e\) at stage~\(s\)) if \(\Phi_{e,s}^{\sigma}(e)\downarrow\) and \(\sigma\) is an initial segment of \(D_s\). 
The node \(\beta\) maintains a local partial function \(\Upsilon=\Upsilon_e\) throughout the construction. 
For a finite string \(\sigma\), \(\Upsilon(\sigma)\) is a mapping that maps a stage~\(s\) to a subset of \(\{i<e\}\). 
If \(L_e\) is initialized at stage \(s\), then \(\Upsilon\) is canceled.
% We say that \(L_e\) is \emph{initialized} at stage \(s\) if there exist some \(i<e\) and some \(j<2(e-i)\) such that \(D_{s-1}(d_{i,j})\neq D_s(d_{i,j})\).

\textbf{\(L_e\)\nbd{}strategy:}
At stage~\(s\) (with \(s^*<s\) the last \(\beta\)\nbd{}stage),
\begin{enumerate}
    \item\label{it:L_mark_unrestorable} For each \(\sigma\in \dom \Upsilon\) with \(s_\sigma\le s^*\) (\(s_\sigma\) is defined in Item~\ref{it:L_record} below) and \(\Upsilon(\sigma)(s^*)\neq \Upsilon(\sigma)(s)\) (\(\Upsilon\) is defined in Item~\ref{it:Upsilon_active} below), we mark \(\tau\in \dom \Upsilon\) as \emph{unrestorable} for those \(\tau\) with \(s_\tau>s_\sigma\).
    \item\label{it:L_up} If \(\Phi_{e,s}^{D_s}(e)\uparrow\), we let \(\beta\concat 0\) act.
    \item If \(\Phi_{e,s}^{D_s}(e)\downarrow\), we let \(\sigma\) be the computation.
    \begin{enumerate}
        \item\label{it:L_restored} If \(\sigma\in \dom \Upsilon\), let \(\beta\concat 0\) act.
        \item\label{it:L_record} If \(\sigma\notin \dom \Upsilon\), let \(s_\sigma=s\) and we mark \(\sigma\) as \emph{restorable}. We initialize all nodes below \(\beta\) and do the following:
        \begin{enumerate}
            \item\label{it:Upsilon_active} For each \(i<e\) with \(d_{i,2(e-i)}\) active, let \(d_{\sigma,i}^*=d_{i,2(e-i)}\), enumerate them into \(D\) and undefine \(d_{i,j}\) for \(j>2(e-i)\). Let
            \[
                \Upsilon(\sigma)(s)=\{i<e \mid D_s(d_{\sigma,i}^*)=1\}.
            \]
            (Here, \(d_{\sigma,i}^*\) records the current version of \(d_{i,2(e-i)}\) at stage~\(s_\sigma\).)
            \item\label{it:Upsilon_defined} For each \(i<e\) with \(d_{i,2(e-i)}\) defined, we let it be undefined (and also let \(d_{i,j}\) be undefined for \(j>2(e-i)\)). \emph{Note that for these \(i\), \(\Gamma_i^W(2(e-i))\uparrow\).}
            \item\label{it:Upsilon_enumerated} For each \(i<e\) with \(d_{i,2(e-i)}\) enumerated, we do nothing.
        \end{enumerate}
    \end{enumerate}
    If \(L_e\) is initialized, we also initialize all nodes below \(\beta\).
\end{enumerate}

The idea is that \(\Upsilon(\sigma)(s)\) collects those \(i\) for which \(R_i\) still pose threats to \(\sigma\) at stage~\(s\). 
When \(\Upsilon(\sigma)(s)=\varnothing\) at some stage~\(s>s_\sigma\) while still restorable, we shall have \(\sigma\) restored and Item~\ref{it:L_restored} happens.
\begin{lemma}\label{lem:L_restored}
    Let \(\beta\) be an \(L_e\)-node. If \(\Upsilon(\sigma)(s)=\varnothing\) for some \(s\) and \(\sigma\) is restorable at stage~\(s\), then \(\sigma\) is a computation for \(L_e\) at stage \(s\).
\end{lemma}
\begin{proof}
    We may assume that \(\beta\) is not initialized during the definition of the current \(\Upsilon\).
    At stage \(s_\sigma\), 
    \begin{itemize}
        \item For \(i\in \Upsilon(\sigma)(s_\sigma)\), let \(s_0\) be the stage when \(d_{\sigma,i}^*\) is extracted in the \(R_i\)\nbd{}strategy, Item~\ref{it:extract}. Then \(s_\sigma<s_0\le s\). Any new version of \(d_{i,2(e-i)}\) defined at stage \(s > s_0\) will be larger than \(\abs{\sigma}\). 
        \item For \(i\notin \Upsilon(\sigma)(s_\sigma)\) and \(d_{i,2(e-i)}\) is undefined, the new version of \(d_{i,2(e-i)}\) defined at stage~\(s>s_\sigma\) will be larger than \(\abs{\sigma}\). 
        \item For \(i\notin \Upsilon(\sigma)(s_\sigma)\) and \(d_{i,2(e-i)}\) is enumerated (in this case, \(d_{i,2(e-i)}=d_{\tau,i}^*\) for some \(\tau\) and \(s_\tau<s_\sigma\)), this poses no threat to \(\sigma\). To see this, consider \(s>s_\sigma\). If \(s\) is an \(R_i\)\nbd{}expansionary stage, then \(d_{i,2(e-i)}\) is extracted and \(\sigma\) is marked as unrestorable, contradicting to the assumption; if \(s\) is not an \(R_i\)\nbd{}expansionary stage, then the \(R_i\)\nbd{}node does nothing.
    \end{itemize}
    Therefore, for \(i<e\), only \(d_{i,x}\) with \(x<2(e-i)\) can injure \(\sigma\). But when this happens, \(\beta\) is initialized and \(\Upsilon\) is canceled.
\end{proof}

\begin{lemma}\label{lem:Upsilon_disjoint}
    Suppose \(\sigma,\tau\in \dom \Upsilon\), \(\sigma\neq \tau\), \(s_\sigma<s_\tau\), and \(\sigma\) is restorable at stage~\(s_\tau\). Then we have \(\Upsilon(\sigma)(s_\tau)\cap \Upsilon(\tau)(s_\tau)=\varnothing\).
\end{lemma}
\begin{proof}
    Let \(i\in \Upsilon(\sigma)(s_\tau)\). Then Item~\ref{it:extract} has not happened for the \(R_i\)\nbd{}strategy at stage~\(s_\tau\). Therefore, \(d_{i,2(e-i)}\) is still enumerated, and \(i\) is in Case~\ref{it:Upsilon_enumerated} for \(\Upsilon(\tau)\) at stage~\(s_\tau\).
\end{proof}
Therefore, while \(i\in \Upsilon(\sigma)(s)\), \(R_i\) poses no threat to \(\tau\).

\begin{lemma}\label{lem:low_bounded}
    Let \(\beta\) be an \(L_e\)-node. Then \(\abs{\dom\Upsilon}\le 2^e\).
\end{lemma}
\begin{proof}
    We may assume that \(\beta\) is never initialized, as \(\Upsilon\) is canceled whenever initialization occurs. We proceed by induction. Let \(f(k)\) denote the maximum number of \(\dom \Upsilon\) for an \(L_k\)-node. Since there are \(k\) many \(R\)\nbd{}nodes above the \(L_k\)\nbd{}node, the \(L_k\)-node only needs to interact with these \(R\)\nbd{}nodes. We aim to prove that \(f(e)\le 2^e\).

    For the base case, \(f(0)=1=2^0\). Let \(\sigma\) be the first computation found by \(\beta\). While \(\abs{\Upsilon(\sigma)(s)}=k\), the \(L_e\)-node handles at most \(e-k\) many \(R\)\nbd{}nodes (Lemma~\ref{lem:Upsilon_disjoint}), which is analogous to the situation faced by an \(L_{e-k}\)-node. Consequently, Item~\ref{it:L_record} occurs at most \(f(e-k)\) times. As \(k\) ranges from \(e\) down to \(1\) (when \(k=0\), we have \(\sigma\) restored), we obtain
    \[
        f(e)\le 1+f(0)+f(1)+\cdots+f(e-1)\le 1+2^0+2^1+\cdots+2^{e-1}=2^e,
    \]
    where the initial \(1\) in the inequality accounts for the first computation \(\sigma\) found by \(\beta\).
\end{proof}

\begin{lemma}\label{lem:low_satisfied}
    For each \(e\in M\), \(\lim_{s}\Delta(e,s)\) exists, and therefore \(D\) is low.
\end{lemma}
\begin{proof}
    Let \(\beta\) be the \(L_e\)-node.
    By Lemma~\ref{lem:finite_injury}, there exists some stage~\(s_0\) after which \(\beta\) is never initialized. By Lemma~\ref{lem:low_bounded}, there exists some stage~\(s_1>s_0\) such that either \begin{itemize}
        \item \(\forall s\ge s_1, \Phi_{e,t}^{D_t}(e)\uparrow\) with Item~\ref{it:L_up} happening; or
        \item \(\forall s\ge s_1, \Phi_{e,t}^{D_t}(e)\downarrow\) with Item~\ref{it:L_restored} happening.
    \end{itemize}
    In the first case, we have \(\Delta(e,s)=0\) for all \(s\ge s_1\). In the second case, we have \(\Delta(e,s)=1\) for all \(s\ge s_1\).
\end{proof}
\subsection{Verification}
We need to find an effective bound for the number of times that each \(\alpha\) is initialized.
\begin{definition}
    Let \(f(e)\) denote the maximum number of \(\abs{\Upsilon_e}\), \(g(e)\) denote \(1+\) the number of times initialization occurs to \(L_e\), and \(h(e)\) denote \(1+\) the number of times initialization occurs to \(R_e\). 
\end{definition}
By Lemma~\ref{lem:low_bounded}, \(f(e)\le 2^e\). 
% \(g(0)= 1\) and \(h(0)\le 2\) as \(R_0\) is initialized if Item~\ref{it:L_record} happens for \(L_0\).

\begin{lemma}\label{lem:finite_injury_prelemma}
    We have the following recurrence relations:
    \[
    \begin{aligned}
        g(0) &= 1, \\
        h(e) &\leq (f(e)+1)g(e), &&\quad e \ge 0, \\
        g(e) &\leq h(e-1)+e,      &&\quad e > 0.
    \end{aligned}
    \]
\end{lemma}
\begin{proof}
    It is clear that \(g(0)=1\) as \(L_0\) is never initialized. For the second inequality, while \(L_e\) is not initialized, Item~\ref{it:L_record} happens at most \(f(e)\) times and \(R_e\) is therefore initialized at most \(f(e)\) times. Therefore, \(h(e)\leq (f(e)+1)g(e)\). For the third inequality, there are \(e\) many \(R\)\nbd{}nodes above. For \(R_i\), if \(d_{i,y}\) for \(y<2(e-i)\) is enumerated or extracted, then \(L_e\) is initialized. Note that if \(y\) is even, then \(R_{e-1}\) is initialized. But if \(y\) is odd, then the enumeration is due to a change of \(K(y)\) from \(0\) to \(1\). As \(K\) is global, \(g(e)\leq h(e-1)+e\).
\end{proof}

\begin{lemma}\label{lem:finite_injury}
    For each \(\alpha\in \mathcal{T}\), there exists some stage~\(s_0\) after which \(\alpha\) is never initialized.
\end{lemma}
\begin{proof}
    Lemma~\ref{lem:finite_injury_prelemma} implies that \(g(e)\le 2 h(e-1)\) and \(h(e)\le 2^{e+1}g(e)\). Then, we have 
    \[
        g(e)\le 2^e \cdot 2^{1+2+\cdots + e}=2^{\frac{e(e+1)}{2}+e}
    \]
    and 
    \[
        h(e)\le 2^e \cdot 2^{1+2+\cdots + e+e+1} = 2^{\frac{(e+1)(e+2)}{2}+e}
    \]
    Since these numbers are effectively bounded, we have the desired result.
\end{proof}

It remains to prove Lemma~\ref{lem:R2}.
\begin{proof}[Proof of Lemma~\ref{lem:R2}]
    We assume that \(R_e\) is not initialized. 
    Fix an \(x\); we estimate the number of times that \(\Gamma_\alpha^{W_e}(x)\) is undefined.
    Equivalently, we estimate the number of times that \(d_{e,y}\) for \(y\le x\) is enumerated into \(D\). 
    Let \(p(x)\) be the number. 
    \(p(0)=0\). 
    Note that each time \(\Gamma_\alpha^{W_e}(y)\) with \(y<x\) is undefined, \(\Gamma_\alpha^{W_e}(x)\) is also undefined. If \(x\) is odd, then \(p(x)\le p(x-1)+1\). If \(x\) is even and \(x=2(k-e)\) for some \(k>e\), then (different versions of) \(d_{e,x}\) can be enumerated into \(D\) for \(L_k\) at most \(2^k\) times. In such a case, \(p(x)\le p(x-1)+2^k=p(x-1)+2^{\frac{x}{2}+e}\). Coarsely estimated, for each \(x\), we have 
    \[
        p(x)\le k + 2^e(1+2+2^2+\cdots +2^k)=2^{e+k+1}-2^{e}+k
    \]
    where \(k=\frac{x}{2}\) if \(x\) is even; \(k=\frac{x-1}{2}\) if \(x\) is odd.
    Therefore, there exists some stage~\(s_0\) after which \(\Gamma_\alpha^{W_e}(x)\) is not initialized.
\end{proof}

Similar to the argument in Section~\ref{subsec:sigma1}, the construction is verifiable in \(P^-+I\Sigma_1\). This completes the proof of Theorem~\ref{theorem:upperisolated}.

\section{Open problems}\label{sec:oq}
In Theorem~\ref{theorem:isolated} and Theorem~\ref{theorem:upperisolated}, we prove over \(P^-+I\Sigma_1\) the existence of an isolated d.c.e.\ degree and an upper isolated d.c.e.\ degree \emph{separately}. Using an infinite injury priority method, Wu~\cite{wu2004} is able to prove the existence of a bi-isolated d.c.e.\ degree. It is then natural to ask
\begin{question}
Does \(P^-+I\Sigma_1\) prove the existence of a bi-isolated d.c.e.\ degree?
\end{question}

Recall from Lemma~\ref{lem:isolate_proper} that over \(I\Sigma_1\), Sacks' Splitting Theorem implies an isolated d.c.e.\ degree must be proper.
In the proof of Theorem~\ref{theorem:upperisolated}, we ensure the d.c.e.\ degree \emph{proper} explicitly with Cooper's requirements~\cite{cooper} and Kontostathis's observation~\cite{kon93}. It would be better if we have an upper density theorem over \(I\Sigma_1\) so that an upper isolated d.c.e.\ degree is automatically proper.
\begin{question}
Is there an upper density theorem for c.e.\ degrees over \(I\Sigma_1\)?
\end{question}

How about the classical Sacks Density Theorem~\cite{sacks1964}? While Groszek, Mytilinaios and Slaman~\cite{groszek1996} showed \(B\Sigma_2\) suffices to prove it, the exact strength remains unclear:

\begin{question}
Does Sacks Density Theorem fail in some \(I\Sigma_1\) model?
\end{question}

%%%%%%%
\section*{Acknowledgments}
%Dedicated to Professor Yue Yang on the occasion of his 60th birthday.
The authors would like to thank Professor Yue Yang for his valuable discussions and insightful suggestions.

\bibliographystyle{plain}
\bibliography{peng.bib}
  
\end{document}